\documentclass[11pt,reqno]{amsart}
\usepackage{color, amsmath,amssymb, amsfonts, amstext,amsthm, latexsym}
\usepackage[active]{srcltx}
\usepackage{extarrows}
\allowdisplaybreaks \textheight 20 true cm
\numberwithin{equation}{section}
\newtheorem{cro}{Corollary}[section]

\newtheorem{prop}{Proposition}[section]
\newtheorem{thm}{Theorem}[section]
\newtheorem{lem}{Lemma}[section]

\begin{document}

\title[Slow entropies and variational principle of subsets ]
{Slow entropy for noncompact sets and variational principle }

\author[D.Kong and E. Chen]
{De-peng Kong and Er-cai Chen }

\address[D.~Kong]
{Institute for the Mathematical Sciences\\
Nanjing Normal University\\
Nanjing, 210046, P.R. China}
\email[D.~Kong]{kongdepengly@163.com}

\address[E.~Chen]
{ Institute for the Mathematical Sciences\\
Nanjing Normal University\\
Nanjing, 210046, P.R. China.
Center of Nonlinear Science\\
Nanjing University\\
Nanjing, 210093, P.R. China } \email[E. ~Chen]{ecchen@njnu.edu.cn}

\thanks{The second author was supported by the National Natural
Science Foundation of China (10971100) and National Basic Research
Program of China (973 Program) (2007CB814805).}

\keywords{topological entropy; topological slow entropy; measure-theoretic slow entropy; variational principle}

\begin{abstract}

This paper defines and discusses the dimension notion of topological slow entropy
of any subset for $\mathbb{Z}^d-$actions. Also, the notion of measure-theoretic slow entropy for $\mathbb{Z}^d-$actions is presented, which is modified from Brin and Katok \cite{BK}. Relations between Bowen topological entropy [3,17] and topological slow entropy are studied in this paper, and several examples of the topological slow entropy in a symbolic system are given. Specifically, a variational principle is proved.

\end{abstract}

\maketitle

\section{Introduction}
Let $(X,\mathcal{T})$ be a topological dynamical system( for short TDS), where $(X,d)$ is a compact
metric space with compatible metric $d$ and $\mathcal{T}$ a continuous $L$-action on $X$.
The set $\mathcal{M}(X)$ denotes the compact convex set of all Borel probability measures, and
$\mathcal{M}(X,\mathcal{T})$ the compact convex set of $\mathcal{T}-$invariant Borel probability measures. We denote by $\mathbb{Z}_+$ the
set of all non-negative integers.

Entropy is one of the most widely used notions in the characterization of the complexity of topological dynamical systems.
Among the notions of entropy, there are two classical ones which are topological entropy
and measure-theoretical entropy. In 1958 Kolmogorov \cite{KL} introduced the definition of
measure-theoretical entropy for an invariant measure and in 1965 Adler \textit{et al} \cite{AK}
defined topological entropy. According to Goodwyn \cite{GY}, Goodman \cite{GA} and Misiurewicz \cite{MI}'s
work, a classical variational principle is found. The topological variational principle
establishes that topological entropy is supremum over all $\mu\in\mathcal{M}(X,T)$ of
the measure-theoretical entropy.

In1973 Bowen \cite{BO} introduced the topological entropy $h_{top}^B(T,K)$ for any set $K$ in a TDS $(X,T)$
resembling Hausdorff dimension. He also proved the remarkable result that
$h_{top}^B(T,G_{\mu})={h}_\mu(T)$ for ergodic measure~$\mu$,
where $G_{\mu}$ denotes the set of \textit{generic} points of $\mu$, and ${h}_\mu(T)$ is
the measure-theoretical entropy. Bowen's topological entropy plays a key role in topological dynamical systems and
dimension theory, see Pesin \cite{PE}. In 1983 Brin and Katok \cite{BK} gave a topological
version of the Shannon-McMillan-Breiman theorem with a local decomposition of
the measure-theoretical entropy. Recently, Feng and Huang \cite{FH} gave a certain variational
relation between Bowen's topological entropy and measure-theoretic entropy for arbitrary
non-invariant compact set. i.e.
\begin{equation*}
h_{top}^B(T,K)=\sup\{\underline{h}_\mu(T):\mu\in \mathcal{M}(X),\mu(K)=1\}
\end{equation*}
where $K$ is any non-empty compact subset of $(X,T)$, $\underline{h}_\mu(T)$ is
 the measure-theoretical lower entropy of Borel probability measure $\mu$ (see \cite{BK,FH}).

The name \textbf{slow entropy} was introduced into dynamical systems by
 Katok and Thouvenot\cite{KT}, Hochman\cite{HO} for $\mathbb{Z}^k$-actions.

In this paper, we use the Carath$\acute{e}$odory dimension structure to study the slow entropy. We define a new topological slow entropy
 $h_{top}^S(\mathcal{T,}Z)$ of any subset
$Z\subseteq X$ for higher dimension $\mathbb{Z}^d-$actions like Hausdorff dimension inspired by Bowen\cite{BO}'s definition of topological
entropy, but it varies or converges more slowly than Bowen's. We also give a modification
of Brin an Katok's definition of measure-theoretical lower entropy, i.e. measure-theoretical slow entropy $\underline{h}_\mu^S(\mathcal{T})$ for higher dimension $\mathbb{Z}^d-$actions. we prove that if Bowen entropy is positive, the topological slow entropy must be infinite.
We prove a  variational principle for slow entropies:
\begin{equation*}
h_{top}^S(\mathcal{T},K)=\sup\{\underline{h}_\mu^S(\mathcal{T}):\mu\in \mathcal{M}(X),\mu(K)=1\},
\end{equation*}
where $K$ is any non-empty compact subset of $X$.

The paper is organized as follows. In Sect. 2 we give the
 definition of topological slow entropy for $\mathbb{Z}^d-$actions in the form of Hausdorff dimension, topological slow entropy using open covers, and some properties. In Sect. 3 some examples in a symbolic dynamical system are given.
In Sect. 4  measure-theoretical slow entropy for $\mathbb{Z}^d-$actions is presented .We prove a variational principle, the topological slow entropy is supermum over all Borel probability measure of
the measure-theoretical  slow entropy.

\section{Slow entropies and related properties}

In this section, we give definitions and some related properties of two slow entropies of subsets in a topological dynamical system: topological slow entropy for $\mathbb{Z}^d-$actions in the form of dimension and using open covers for $\mathbb{Z}_+-$action.

Firstly, we introduce the notion of $L$-action found in \cite{YA} for convenience. Let $(X,\mathcal{T})$ be a TDS.
 A family of continuous transformations
$\mathcal{T}=\{T^h: X\rightarrow X\}_{h \in L}$ is called  a continuous $L$-action, with $L=\mathbb{Z}^d$ or $\mathbb{Z}_+^d$,$d\geq1$,
if $\mathcal{T}$ satisfies $T^{h+k}=T^h\circ T^k, h,k\in L$,and $T^0$ is the identity map.
For $k\in L$ and $H\subset L$, we set $H+k=\{h+k: h\in H\}$. For $n\in \mathbb{Z}_+$, let
$$H_n:=\{h=(h_1,h_2,\cdots,h_d)\in L:|h_i|<n,1\leq i\leq d\},$$
and $\lambda_n:=\sharp H_n$,where $\sharp G$ denotes the cardinality of the set $G$.

Secondly, the definition of Bowen ball for $\mathbb{Z}^d-$actions is as follow: For
$n\in \mathbb{N},x\in X,\epsilon>0$, we denote by
$B_{n}(x,\epsilon)$ the open Bowen ball of radius $\epsilon>0$ in the metric $d_{n}$ around
$x,$ i.e.
\begin{equation*}
    d_n(x,y)=\max_{h\in H_n}d(T^hx,T^hy),
\end{equation*}
\begin{equation*}
B_n(x,\epsilon)=\{y \in X:d_n(x,y)< \epsilon\}.
\end{equation*}

\subsection{Dimension definition of topological slow entropy.}

We now give the definition of slow entropy with $\mathbb{Z}^d$ or $\mathbb{Z}_+^d$-actions. Let $(X,d)$ be a compact metric space
and $\mathcal{T}$ be a continuous $L$-action on $X$ with $L=\mathbb{Z}^d$ or $\mathbb{Z}_+^d$, $d\geq1$.
For $Z\subset X,s\geq0,N\in\mathbb{N},\epsilon>0$. Define
\begin{equation*}
M^\mathcal{T}(Z,s,N,\epsilon)=\inf_{\mathcal{G}}\bigg\{\sum\limits_{i}\bigg(\frac{1}{\lambda_{n_i}}\bigg)^s\bigg\},
\end{equation*}
where the infimum is taken over all finite or countable families $\mathcal{G}:=\{B_{n_i}(x_i,\epsilon)\}$ such that
$x_i\in X, n_i\geq N$ and $\bigcup_{i}B_{n_i}(x_i,\epsilon)\supset Z$.
Clearly, $M^\mathcal{T}(Z,s,N,\epsilon)$ does not decrease as $N$ increases and $\epsilon$ decreases , hence the
following limit exists:
\begin{equation*}
M^\mathcal{T}(Z,s,\epsilon)=\lim\limits_{N\rightarrow\infty}M^\mathcal{T}(Z,s,N,\epsilon),
\end{equation*}
then we can easily see that there exists the critical value $h^S_{top}(\mathcal{T},Z,\epsilon)\geq0$ such that
\begin{equation*}
 h^S_{top}(\mathcal{T},Z,\epsilon)=\inf\{s:M^\mathcal{T}(Z,s,\epsilon)=0\}=\sup\{s:M^\mathcal{T}(Z,s,\epsilon)=\infty\}.
\end{equation*}
Finally we set

\begin{equation*}
h^S_{top}(\mathcal{T},Z)=\lim\limits_{\epsilon\rightarrow0}h^S_{top}(\mathcal{T},Z,\epsilon),
\end{equation*}
which is called the\textit{ topological slow entropy } for $Z$ with respect to $\mathcal{T}$.

This quantity $h^S_{top}(\mathcal{T},\bullet)$ is defined in way which resembles the Hausdorff dimension, also satisfies most properties like Bowen entropy\cite{BO}
for $\mathbb{Z}_+$ action.  For convenience, we use ~$M(Z,s,N,\epsilon)$, $M(Z,s,\epsilon)$~instead of~$M^\mathcal{T}(Z,s,N,\epsilon)$, $M^\mathcal{T}(Z,s,\epsilon)$ without any confusion.

\begin{prop}

\noindent{\rm (i)}If $d=1$, then $h^S_{top}(T^m,Z)=h^S_{top}({T},Z),m>0;$

\noindent{\rm (ii)}If $Z_1\subset Z_2\subset X$, then $h^S_{top}(\mathcal{T},Z_1)\leq h^S_{top}(\mathcal{T},Z_2);$

\noindent{\rm (iii)} $h^S_{top}(\mathcal{T},\bigcup_{i=1}^{\infty}Z_i)=\sup_{i}h^S_{top}(\mathcal{T},Z_i)$; If $Z\subset X$ is a countable set, then $h^S_{top}(\mathcal{T,}Z)=0.$

\end{prop}

\begin{proof}
We will only prove Proposition (i), others' proof  are omitted, because we can get them  from the definition without difficulty.
Suppose a finite or countable family $\mathcal{G}=\{ B_{n_i^T}(x_i,\epsilon)\},x_i\in X,n_i^T\geq N$, and $\bigcup_i B_{n_i^T}(x_i,\epsilon)\supset Z$ corresponding to~$T$~in the situation of $d=1$.
Suppose
\begin{eqnarray*}
B_{n_i^{T^m}}(x_i,\epsilon)&=&
\{y\in X: \text{max}_{0\leq j<n_i^{T^m}}d(T^{mj}x_i,T^{mj}y)<\epsilon\}\\
&= & \bigcap_{j=0}^{n_i^{T^m}-1}T^{-mj}B(T^{mj}x_i,\epsilon).
\end{eqnarray*}
We may suppose $n_i^{T^m}\geq N$ as well, then $$n_i^{T^m}= \lceil\frac{n_i^T}{m}\rceil\geq \{N,\frac{n_i^T}{m}-\frac{m-1}{m}\},$$
where $\lceil a\rceil$ denotes the integral part of a real number $a$. Obviously,
$$B_{n_i^{T^m}}(x_i,\epsilon)\supset \bigcap_{p=0}^{n_i^T-1}T^{-p}B(T^px_i,\epsilon)= B_{n_i^T}(x_i,\epsilon),$$
which implies ~$\bigcup_i B_{n_i^{T^m}}(x_i,\epsilon)\supset Z$. Then
\begin{eqnarray*}
\sum_i\bigg(\frac{1}{n_i^{T^m}}\bigg)^s&\leq&
\sum_i\bigg(\frac{m}{n_i^T-(m-1)}\bigg)^s\bigg(\frac{n_i^T}{n_i^T}\bigg)^s\\
&\leq&m^s\bigg(\frac{N}{N-(m-1)} \bigg)^s\sum_i\bigg(\frac{1}{n_i^{T}}\bigg)^s.
\end{eqnarray*}
Taking the infimum of both sides, and we get
\begin{equation*}
    M^{T^m}(Z,s,N,\epsilon)\leq m^s\bigg(\frac{N}{N-(m-1)} \bigg)^sM^T(Z,s,N,\epsilon).
\end{equation*}
By letting $N\rightarrow +\infty$, then $ M^{T^m}(Z,s,\epsilon)\leq m^s M^T(Z,s,\epsilon)$, and then
$$h^S_{top}(T^m,Z,\epsilon)\leq h^S_{top}(T,Z,\epsilon).$$
Letting $\epsilon\rightarrow0$, we have $h^S_{top}(T^m,Z)\leq h^S_{top}(T,Z)$.

Next, we will prove the reverse side. Because $T$ is uniformly continuous, for $\forall \epsilon>0, \exists \delta>0$~ such that
\begin{equation*}
    d(x,y)<\delta \Rightarrow \text{max}_{0\leq j\leq m-1}d(T^jx,T^jy)<\epsilon.
\end{equation*}
We suppose Bowen family $\mathcal{G}=\{B_{n_i^{T^m}}(x_i,\delta)\}$, $x_i\in X, n_i^{T^m}\geq N$ corresponding to~$T^m$~, and
 $\bigcup_{B\in \mathcal{G}}B\supset Z$, then
\begin{equation*}
   \bigcup_{i}B_{mn_i^{T^m}}(x_i,\epsilon)\supset Z.
\end{equation*}
In fact, for any arbitrary $z\in Z$, there exists some ${n_{i_0}^{T^m}}\geq N$ such that ~$z\in B_{n_{i_0}^{T^m}}(x_{i_0},\delta)$, i.e.
\begin{equation*}
   d(T^{mk}z,T^{mk}x_{i_0})<\delta, \ \ 0\leq k\leq n_{i_0}^{T^m}-1.
\end{equation*}
So \begin{equation*}
    d(T^{m{k}+j}z,T^{m{k}+j}x_{i_0})<\epsilon, \ \  0\leq k\leq n_{i_0}^{T^m}-1, \ \ 0\leq j\leq m-1.
\end{equation*}
then $z\in B_{mn_{i_0}^{T^m}}(x_{i_0},\epsilon)$, and then $\bigcup B_{mn_{i}^{T^m}}(x_{i},\epsilon)\supset Z$.

We notice that  any finite or countable family ~$\mathcal{G}^{'}=\{B_{n_j^T}(x_j,\epsilon)\}$ covering $Z$  contains $\{B_{mn_i^{T^m}}(x_i,\epsilon)\}$. So we get
\begin{equation*}
    \inf_{\mathcal{G}}\sum_i\bigg(\frac{1}{n_i^{T^m}} \bigg)^s= m^{-s}\inf_{\mathcal{G}}\sum_i\bigg(\frac{1}{mn_i^{T^m}} \bigg)^s\geq m^{-s}\inf_{\mathcal{G}^{'}}\sum_j\bigg(\frac{1}{n_j^{T}} \bigg)^s.
\end{equation*}
and
\begin{equation*}
    M^{T^m}(Z,s,N,\delta)\geq m^{-s}M^T(Z,s,mN,\epsilon).
\end{equation*}
Letting $N\rightarrow +\infty$, hence $ M^{T^m}(Z,s,\delta)\geq m^{-s} M^T(Z,s,\epsilon)$. And noticing $\delta,\epsilon$ arbitrary,
 we get $h^S_{top}(T^m,Z)\geq h^S_{top}(T,Z)$.

\end{proof}
We point out $h^B_{top}(T,\bullet)$ Bowen topological entropy defined by using Bowen balls. For details, see[18, Page 74].
\begin{prop}
\noindent{\rm (i)} For $\mathbb{Z}_+$-action, $Z\subset X$, $h^S_{top}(T,Z)\geq h^B_{top}(T,Z);$

\noindent{\rm (ii)} For $\mathbb{Z}_+$-action, if $h^B_{top}(T,Z)>0$, then $h^S_{top}(T,Z)=+\infty$.
\end{prop}

\begin{proof}
(i) is obvious. We only prove (ii). We suppose $\forall \beta >0, h_{top}^B(T,Z)=a>0$. From the definition of Bowen entropy,
for arbitrary $\delta>0$, there exists $\epsilon_0>0$, for arbitrary $0<\epsilon<\epsilon_0$, we have
\begin{equation*}
    0<a-\delta< h_{top}^B(T,Z,\epsilon)<a+\delta,
\end{equation*}
which implies $M^B(Z,a-\delta,\epsilon)=+\infty$. Because $M^B(Z,a-\delta,N,\epsilon)$ increases as $N$, so
\begin{equation*}
   M^B(Z,a-\delta,N,\epsilon)\rightarrow+\infty,~ as ~N\rightarrow\infty.
\end{equation*}
Then for any subfamily $\mathcal{G}$ of $Z$, $\sum_i e^{-n_i(a-\delta)}\rightarrow+\infty$. For arbitrary real number $K>0$,
because $\frac{e^{-n_i(a-\delta)}}{(\frac{1}{n_i})^K}=\frac{n_i^K}{e^{n_i(a-\delta)}}\rightarrow 0$, i.e.
For arbitrary $\varepsilon_1>0$, there is some $n^{'}$, for $n_i>n^{'}$,
$\bigg(\frac{1}{{n_i}}\bigg)^K>\frac{1}{\varepsilon_1}e^{-n_i(a-\delta)}$.
Hence $\sum_i\bigg(\frac{1}{n_i}\bigg)^K\rightarrow+\infty$. And thus
\begin{equation*}
   \inf_{\mathcal{G}}\bigg\{\sum\limits_{i}\bigg(\frac{1}{{n_i}}\bigg)^K\bigg\}\rightarrow+\infty,
\end{equation*}
that is $M^S(Z,K,N,\epsilon)\rightarrow+\infty$~as~$N\rightarrow\infty$. Letting $N\rightarrow +\infty$, we have $M^S(Z,K,\epsilon)=+\infty.$
So \begin{equation*}
    h^S(T,Z,\epsilon)\geq K.
\end{equation*}
Because $\epsilon$ and $K$ are arbitrary, hence $h^S(T,Z)=+\infty$. Thus we complete the proof.

\end{proof}

We now give a equivalent definition of $h^S_{top}({T},\bullet)$ for the situation $d=1$, which comes from Bowen \cite{BO}. We set $(X,T)$ be a TDS,
$\mathcal{U}\in C_X^o$  a finite open cover of $X$. $Z\subset X,K\subset X$, let
\begin{equation*}
n^T_{\mathcal{U}}(K)=
\begin{cases}
0& \text{if ~$K\nsucceq \mathcal{U}$},\\
+\infty & \text{ if~$T^iK\succeq \mathcal{U}$,~for all~$i\in \mathbb{Z}_+$},\\
k& \text{$k=\max\{j\in \mathbb{N}: T^i K\succeq \mathcal{U},i =0,1,\cdots,j-1$\}~ ohterwise}.
\end{cases}
\end{equation*}
For $k\in \mathbb{N}$, we define
\begin{equation*}
   \mathfrak{ G}(T,\mathcal{U},Z,k)=\{\mathcal{E}:\mathcal{E}\text{~is~}\text{a countable family of subsets of X,}Z\subseteq \bigcup\mathcal{E}, \mathcal{E}\succeq\mathcal{U}_0^{k-1} \}.
\end{equation*}
Then for each $s\in \mathbb{R}$, set
\begin{equation*}
    m_{T,\mathcal{U}}(Z,s,k)=\inf_{\mathcal{E}\in \mathfrak{G}(T,\mathcal{U},Z,k)}\sum_{E\in \mathcal{E}}\bigg(\frac{1}{n^T_{\mathcal{U}}(E)}\bigg)^s.
\end{equation*}
We write $m_{T,\mathcal{U}}(Z,s,k)=0$ for the case $\emptyset=\mathcal{E}\in \mathfrak{G}(T,\mathcal{U},Z,k)$ by convention. When $k\rightarrow\infty$,
$m_{T,\mathcal{U}}(Z,s,k)$ increases, therefore we could define
\begin{equation*}
    m_{T,\mathcal{U}}(Z,s)=\lim_{k\rightarrow\infty} m_{T,\mathcal{U}}(Z,s,k).
\end{equation*}
We notice that if $s_1\geq s_2$, then $m_{T,\mathcal{U}}(Z,s_1)\leq m_{T,\mathcal{U}}(Z,s_2)$. We define
\textit{Bowen topological slow entropy} $h^{BS}_{\mathcal{U}}(T,Z)$ of $\mathcal{U}$ with respect to ~$T$~ as follows:
\begin{equation*}
 h^{BS}_{\mathcal{U}}(T,Z)=\inf\{s:m_{T,\mathcal{U}}(Z,s)=0\}=\sup\{s:m_{T,\mathcal{U}}(Z,s)=\infty\},
\end{equation*}
and \textit{Bowen topological slow entropy} of $Z$ as follows:
\begin{equation*}
 h^{BS}_{top}(T,Z)=\sup_{\mathcal{U}\in \mathcal{C}_X^o}h^{BS}_{\mathcal{U}}(T,Z).
\end{equation*}

\begin{prop}
\begin{equation*}
h_{top}^{BS}(T,Z)=\lim_{\delta\rightarrow0}h_{top}^S(T,Z,\delta).
\end{equation*}
\end{prop}
\begin{proof}
It can be similarly proved using techniques in \cite{PE}.

\end{proof}

\subsection{Definition of topological slow entropy using open covers.}
We set $(X,T)$ be a TDS,
$\mathcal{U}\in C_X^o$  a finite open cover of $X$. Set $N(\mathcal{U},Z)$  to be the minimal cardinality of
sub-families $\mathcal{V}\subset\mathcal{U}$ with $\bigcup \mathcal{V}\supset Z$, where $\bigcup \mathcal{V}=\bigcup_{V\in \mathcal{V}}V$.
And we write $N(\mathcal{U},\emptyset)=0$. Obviously, $N(\mathcal{U},Z)=N(T^{-1}\mathcal{U},Z)$. Let
\begin{equation*}
    h_{\mathcal{U}}^S(T,Z)=\limsup_{n\rightarrow \infty}\frac{1}{\log n}\log N(\mathcal{U}_0^{k-1},Z).
\end{equation*}
$h_{\mathcal{U}}(T,Z)$ increases with respect to $\mathcal{U}$. We define the topological slow entropy of $Z$ by
\begin{equation*}
    h^S(T,Z)=\sup_{\mathcal{U}\in C_X^o} h_{\mathcal{U}}^S(T,Z).
\end{equation*}

\noindent{\textbf{Remark}}: From the definition of topological slow entropy and topological entropy using open covers, there exists a relation:
\begin{equation*}
    \frac{1}{n}\log N(\mathcal{U}_0^{k-1},Z)=\frac{1}{\log n}\log N(\mathcal{U}_0^{k-1},Z)\cdot \frac{\log n}{n}.
\end{equation*}
Noticing $\frac{\log n}{n}\rightarrow 0$ as $n\rightarrow\infty$. Therefore, if $h_{top}(T,Z)>0$, then $h^S(T,Z)=+\infty$; if $h^S(T,Z)<+\infty$, then
$h_{top}(T,Z)=0$.
\begin{prop}
\begin{equation*}
  h_{\mathcal{U}}^{BS}(T,Z)\leq h^S_{\mathcal{U}}(T,Z),\text{and then}\ \  h_{top}^{S}(T,Z)\leq h^S(T,Z).
\end{equation*}
\end{prop}

\begin{proof}
We only prove~$h_{\mathcal{U}}^{BS}(T,Z)\leq h^S_{\mathcal{U}}(T,Z)$.
Suppose $$\mathcal{A}_n=\{U_{i_0}\cap T^{-1}U_{i_1}\cap\cdots\cap T^{-n+1}U_{i_{n-1}}:U_{i_{k}}\in \mathcal{U}\},$$
such that $\bigcup\mathcal{A}_n\supseteq Z$. Obviously, $A\in \mathcal{A}_n$, $ n_{T,\mathcal{U}}(A)\geq n, s>0$ and we get
\begin{equation*}
    m_{T,\mathcal{U}}(Z,s,n)\leq \sum_{A\in \mathcal{A}_n}\bigg(\frac{1}{n_{T,\mathcal{U}}(A)}\bigg)^s\leq \sum_{A\in \mathcal{A}_n}\bigg(\frac{1}{n}\bigg)^s=N(\mathcal{U}_0^{k-1},Z)\bigg(\frac{1}{n}\bigg)^s.
\end{equation*}
So
\begin{eqnarray*}
m_{T,\mathcal{U}}(Z,s)&\leq&
\limsup_{n\rightarrow\infty}N(\mathcal{U}_0^{k-1},Z)\bigg(\frac{1}{n}\bigg)^s\\
&=& \limsup_{n\rightarrow\infty}\exp(-\log n(s-\frac{1}{\log n}\log N(\mathcal{U}_0^{k-1},Z))).
\end{eqnarray*}
And if $s>\frac{1}{\log n}\log N(\mathcal{U}_0^{k-1},Z)$, then $m_{T,\mathcal{U}}(Z,s)=0$. Therefore,
\begin{equation*}
    h_{\mathcal{U}}^{BS}(T,Z)\leq h^S_{\mathcal{U}}(T,Z).
\end{equation*}

\end{proof}


\section{Examples in a symbolic dynamical system.}
 We take examples with a continuous $L=\mathbb{Z}^d_+$-action  to explain the new
topological slow entropy in a symbolic dynamical system with a special
metric.

Suppose a finite alphabet
$A=\{1,2.\cdots,k\}$, where $k\geq2$. $$A^{L}=\{1,2.\cdots,k\}^{L}=\{(\omega_h)_{h\in L}:\omega_h\in A,h\in L\}.$$
Suppose $\omega,\omega'\in A^L$, let $$n(\omega,\omega')=\min\{k:\omega_h=\omega'_h (h\in H_{k-1}), \omega_h\neq \omega'_h~(\text{for some}~h\in H_k\backslash H_{k-1})\}.$$
Endow $A^{L}$ with the metric
$d(\omega,\omega')=\frac{1}{\lambda_{n(\omega,\omega')}}$ for $\omega,
\omega'\in A^L$, then $d(\bullet,\bullet)$ is a compatible metric for $A^L$.
$A^{L}$ is the onesided
full shift of $k$ symbols, i.e. for $h\in L$, we define the shift action~$\sigma^h:A^L\rightarrow A^L$ as
\begin{equation*}
    (\sigma^h(\omega))_k=\omega_{h+k},k\in L
\end{equation*}
and then $\mathcal{T}=\{\sigma^h\}_{h\in L}$ is a continuous $L$ -action on $A^L$.

\begin{prop}
For any subset $Z\subseteq A^{{L}}$,
$h^S_{top}(\mathcal{\sigma},Z)=\text{dim}_H Z$, where $\text{dim}_H Z$ is denoted
the Hausdorff dimension in $(X,d)$(see\cite{MA}).
\end{prop}

\begin{proof}
In fact, for $m\in \mathbb{N},\omega\in A^L$ we
set a cylinder set as $$C_m(\omega)=\{\omega'\in A^L: \omega_h=\omega'_h,h\in H_m\}.$$
It is not hard to see the $s-$ Hausdorff outer measure of $Z$ can be
$$H(Z,s)=\lim_{\delta\rightarrow0}\inf_{\mathcal{G}}\text{diam}(C_{m_i}(\omega_i))^s,$$
and where the infimum is taken over all finite or countable family $\mathcal{G}:=\{C_{m_i}(\omega_i)\}$, which
covers $Z$ with $\sup_i\text{diam}(C_{m_i}(\omega_i))<\delta$. Let $\epsilon>0$ be sufficiently small and
choose  $n\in \mathbb{N}$ such that $\frac{1}{\lambda_{n+1}}\leq\epsilon<\frac{1}{\lambda_{n}}$. So it also
follows from the choice of the metric $d(\bullet,\bullet)$ that
$B_k(\omega,\epsilon)=C_{k+n-1}(\omega)$ for all $k\in \mathbb{Z}_+$ and $\text{diam}(C_{j}(\omega))=\lambda_{j+1}^{-1}$. Comparing the two definitions, we get the result.

\end{proof}

\begin{prop}
For any real number $0\leq t <+\infty$, there exists a compact subset $E\subset A^{L}$, such that
$h^S_{top}(\mathcal{\sigma},E)=t$.
\end{prop}

\begin{proof}
From Mattila\cite{MA}, we know the fact: If the Hausdorff dimension of a set $A$ is $s$, then for any $0<t<s$, there exists a  compact subset $E$ such that $\text{dim}_H E=t$. From theorem 3.1, we get $h^S_{top}(\mathcal{T},E)=\text{dim}_H E$, this means $h^S_{top}(\mathcal{T},E)=t$.

\end{proof}

\begin{prop}
 In the symbolic system $(A^{\mathbb{Z}_+},\sigma,d)$, for any non-empty subset $Z\subset A^{\mathbb{Z}_+}$, $\overline{\text{dim}}_B Z= h^S(\sigma,Z)$, where $\overline{\text{dim}}_B Z$ is denoted the upper Box dimension of $Z$ (see \cite{MA}).

\end{prop}

\begin{proof}
Suppose $Z\subset A^{\mathbb{Z}_+}$, for $0<\epsilon<+\infty$, let $N(Z,\epsilon)$ be the smallest number of $\epsilon-$ball needed to cover $Z$:
\begin{equation*}
   N(Z,\epsilon)=\min\{k:Z\subset \bigcup^k_{i=1}B(\omega_i,\epsilon) \text{~for some}~\omega_i\in A^{\mathbb{Z}_+}\}.
\end{equation*}

Suppose $\mathcal{U}=\{A_1,\cdots,A_k\}$ be a generator, $\omega=(x_0,x_1,\cdots)$, $$A_j=\{\omega: x_0=j\},$$ and every $A_j$ is clopen set for $j=1,\cdots,k$.
 Suppose $\frac{1}{m+1}\leq \epsilon<\frac{1}{m}$ for $m>1$, for any $\omega_i\in  A^{\mathbb{Z}_+}$, Bowen ball $B_n(\omega_i,\epsilon)$ with metric $d(\bullet,\bullet)$ be the cylinder set as follows:
\begin{equation*}
   C_{n+m-1}(\omega_i)= B_n(\omega_i,\epsilon), \text{and}~B_n(\omega_i,\epsilon)\subset B(\omega_i,\epsilon).
\end{equation*}
Noticing $C_{n+m-1}(\omega_i)=\bigcap_{j=0}^{n+m-1}\sigma^{-j}A_{i_j}\in \mathcal{U}_{0}^{n+m-1}$, the union of elements of $\mathcal{U}_{0}^{n+m-1}$ covers $Z$, then
\begin{equation}\label{2.3}
   N(Z,\epsilon)\leq N(\mathcal{U}_{0}^{n+m-1},Z).
\end{equation}
Similarly, $B(\omega_i,\epsilon)\subset C_{m-1}(\omega_i)=\bigcap_{j=0}^{m-1}\sigma^{-j}A_{i_j}\in \mathcal{U}_{0}^{m-1}$, and then
\begin{equation}\label{2.3}
    N(\mathcal{U}_{0}^{m-1},Z)\leq N(Z,\epsilon).
\end{equation}

From (3.1),(3.2), therefore,
\begin{equation*}
 \frac{\log (m-1)}{\log (m+1)} \cdot\frac{\log N(\mathcal{U}_{0}^{m-1},Z)}{\log (m-1)}\leq\frac{\log N(Z,\epsilon)}{-\log \epsilon}\leq\frac{\log N(\mathcal{U}_{0}^{n+m-1},Z)}{\log (n+m-1)}\cdot\frac{\log (n+m-1)}{\log m}.
\end{equation*}
Taking $n,m\rightarrow +\infty$, $\epsilon\rightarrow 0$ and supper limit,  we get
\begin{equation*}
   \overline{\text{dim}}_B Z= h^S(\sigma,Z).
\end{equation*}
\end{proof}

If $(X,d)$ is a compact metric space, $Z\subset X$, $\{U_\alpha\}_{\alpha\geq1}$ is a $\epsilon-$cover of $Z$, if $\mid U_\alpha \mid\leq\epsilon$ for $\epsilon>0$ and $\bigcup_\alpha U_\alpha\supset Z$, where $\mid U \mid$ denotes the diameter of $U$.

\begin{lem}
If $(X,d)$ and $(Y,\rho)$ are both compact metric spaces, $f:X\rightarrow Y$ is a map,  and if there exists  $\delta_0>0$, so that $\forall ~0<\delta\leq\delta_0$, $\mid B\mid \leq \delta$ for $B\subset X$, $f\mid_B $ is a bi-Lipschitz map i.e.
\begin{equation*}
    c_1 d(x,y)\leq \rho (f(x),f(y))\leq c_2 d(x,y) ~~~~~~~for ~~\forall x,y\in B, c_1,c_2>0.
\end{equation*}
Then for any $Z\subset X$, we have $dim_H(f(Z))=dim_H(Z)$.

\end{lem}

\begin{proof}
Suppose $\{U_i\}_{i\geq1}$ is a $\delta-$cover of $Z$, since $f\mid_{U_i}$ is a bi-Lipschitz map for aibitray $i$, so
\begin{equation*}
    \mid f(Z\cap U_i)\mid\leq c_2 \mid Z\cap U_i\mid\leq c_2\delta,
\end{equation*}
then $\{f(Z\cap U_i)\}$ is a $\epsilon:=c_2\delta-$cover of $f(Z)$. For $\forall s>0$,
\begin{equation*}
    \sum_i\mid f(Z\cap U_i)\mid^s\leq c_2^s\sum_i\mid Z\cap U_i\mid^s\leq  c_2^s\sum_i\mid U_i\mid^s,
\end{equation*}
which implies that
\begin{equation*}
    H^s_\epsilon(f(Z))\leq H^s_\delta(Z).
\end{equation*}
Taking $\delta,\epsilon\rightarrow0$, we get $H^s(f(Z))\leq H^s(Z)$.
If $s>\text{dim}_H Z$, then $H^s(f(Z))\leq H^s(Z)=0$, that is $\text{dim}_H f(Z)\leq s$ for all $s>\text{dim}_H Z$, so $\text{dim}_H f(Z)\leq \text{dim}_H Z$.

Noticing the bi-Lipschitz mapping $f\mid_{B}$  has a inverse mapping $f^{-1}\mid_B:f(B)\rightarrow B$ and using the above result, we have $\text{dim}_H f(Z)=\text{dim}_H Z$.

\end{proof}

\begin{prop}
In the symbolic system $(A^{\mathbb{Z}_+},\sigma,d)$, for arbitray $t>0$, there exists a $F_\sigma$ subset $E$, $\sigma E \subset E$, such that $h^S_{top}(\sigma,E)= t$.

\end{prop}

\begin{proof}
From Proposition 3.2 and Lemma 3.1, for any $t>0$, there is a compact subset $Z\subset A^{\mathbb{Z}_+}$, $\textmd{dim}_H \sigma Z=\textmd{dim}_H Z=h^S_{top}(\sigma,Z)= t$. We set $E=\bigcup^\infty_{i=0}\sigma^iZ$, then $E$ is a $F_\sigma$ subset, $\sigma E \subset E$, and $\textmd{dim}_H E=\textmd{dim}_H Z$, which implies $h^S_{top}(\sigma,E)=t$.

\end{proof}


\section{The variational principle for slow entropies.}
 We firstly introduce the measure-theoretic slow entropy. The notion of weighted topological slow entropy is presented, which is important to prove the variational principle.
\subsection{Definition of measure-theoretic slow entropy}

From [2], Brin and Katok defined the local or measure-theoretic entropy for $\mathbb{Z}_+-$action as follows:
Suppose $\mu\in\mathcal{M}(X)$, define
\begin{equation*}
\underline{h}_\mu(T,x)=\lim\limits_{\epsilon\rightarrow0}\liminf\limits_{n\rightarrow\infty}-\frac{1}{n}\log \mu(B_n(x,\epsilon));
\underline{h}_\mu(T,x)=\lim\limits_{\epsilon\rightarrow0}\liminf\limits_{n\rightarrow\infty}-\frac{1}{n}\log \mu(B_n(x,\epsilon)).
\end{equation*}

\begin{equation*}
\underline{h}_\mu(T)=\int\underline{h}_\mu(T,x)d\mu(x);\overline{h}_\mu(T)=\int\overline{h}_\mu(T,x)d\mu(x).
\end{equation*}
They also proved the proposition:
For $\mu\in\mathcal{M}(X,T),\mu-$a.e.$x$, $\underline{h}_\mu(T,x)=\overline{h}_\mu(T,x)$, and
$\int\underline{h}_\mu(T,x)d\mu(x)=h_{\mu}(T).$
Hence for $\mu\in\mathcal{M}(X,T)$, $\underline{h}_\mu(T)=\overline{h}_\mu(T)=h_{\mu}(T).$

Now, we give a modification of measure-theoretic lower entropy of higher dimension $\mathbb{Z}^d$-actions.
Suppose $\mu\in\mathcal{M}(X)$, define
\begin{equation*}
\underline{h}_\mu^S(\mathcal{T},x)=\lim\limits_{\epsilon\rightarrow0}\liminf\limits_{n\rightarrow\infty}-\frac{1}{\log \lambda_n}\log \mu(B_n(x,\epsilon));
\end{equation*}

\begin{equation*}
\underline{h}_\mu^S(\mathcal{T})=\int\underline{h}_\mu^S(\mathcal{T},x)d\mu(x).
\end{equation*}
We call $\underline{h}_\mu^S(\mathcal{T},x)$ the \textit{measure-theoretic slow entropy }of point $x$ with respect to $\mathcal{T}$,
and $\underline{h}_\mu^S(\mathcal{T})$ the \textit{measure-theoretic slow entropy} of $X$ with respect to $\mathcal{T}$.

\noindent{\textbf{Remark}}: From the definition of measure-theoretic slow entropy, it is easy to know
$$-\frac{1}{\lambda_n}\log \mu(B_n(x,\epsilon))=-\frac{1}{\log \lambda_n}\log \mu(B_n(x,\epsilon))\cdot \frac{\log \lambda_n}{\lambda_n},$$
because $\frac{\log \lambda_n}{\lambda_n}\rightarrow0$ as $n\rightarrow\infty$. If $\underline{h}_\mu^S(\mathcal{T})$  be finite, then~$h_{\mu}(\mathcal{T})=0$;
if $h_{\mu}(\mathcal{T})>0$, $\underline{h}_\mu^S(\mathcal{T})$ must be infinite.
\subsection{Weighted topological slow entropy.}
For any positive function $f:X\rightarrow [0,\infty), N\in\mathbb{N},\epsilon>0$, we define
\begin{equation*}
W(f,s,N,\epsilon)=\inf\sum_{i}c_i\bigg(\frac{1}{\lambda_{n_i}}\bigg)^s,
\end{equation*}
where the infimum is taken over all finite or countable families $\{(B_{n_i}(x_i,\epsilon),c_i)\}$ such that
$x_i\in X, n_i\geq N,0<c_i<\infty$ and $\sum_ic_i\chi_{B_i}\geq f$.

For $Z\subset X,f=\chi_{Z}$, set $W(Z,s,N,\epsilon)=W(\chi_Z,s,N,\epsilon)$. Clearly, the function $W(Z,s,N,\epsilon)$
does not decrease as $N$ increases and $\epsilon$ decreases. So the following limits exist:
\begin{equation*}
W(Z,s,\epsilon)=\lim\limits_{N\rightarrow\infty}W(Z,s,N,\epsilon), W(Z,s)=\lim\limits_{\epsilon\rightarrow0}W(Z,s,\epsilon).
\end{equation*}

It's not difficult to prove that there exists a critical value of parameter $s$, which we will denote by $h_{top}^W(\mathcal{T},Z)$, such that
\begin{equation*}
W(Z,s)=\left\{
                            \begin{array}{ll}
                              0, & s>h_{top}^W(\mathcal{T},Z) ;\\
                              \infty, & s<h_{top}^W(\mathcal{T},Z).
                            \end{array}
                          \right.
\end{equation*}
We call $h_{top}^W(\mathcal{T},Z)$ the\textit{ weighted topological slow entropy} of $Z$ with respect to $\mathcal{T}$.


\subsection{Equivalence of $h_{top}^S$ and $h_{top}^W$.}

\begin{prop}
\noindent{\rm (i)} For any $s\geq0,N\in\mathbb{N},\epsilon>0,M(\cdot,s,N,\epsilon) \text{\ and\ } W(\cdot,s,N,\epsilon)$
 are outer measures on $X$.

\noindent{\rm (ii)}For any $s\geq0$, both $M(\cdot,s)$ and $W(\cdot,s)$ are metric outer measures on $X$.
\end{prop}

\begin{prop}
Suppose $Z\subset X$, for any $s\geq0,\epsilon,\delta>0$, we have $$M(Z,s+\delta,N,6\epsilon)\leq W(Z,s,N,\epsilon)\leq M(Z,s,N,\epsilon)$$
for large enough $N$. And then $h_{top}^S(\mathcal{T},Z)=h_{top}^W(\mathcal{T},Z)$.
\end{prop}

\begin{lem}\cite{MA}
Let $(X,d)$ be a compact metric space and $\mathcal{B}=\{B(x_i,r_i)\}_{i\in\mathcal{I}}$ be a family open of (or closed) balls
in $X$. Then there exists a finite or countable subfamily $\mathcal{B}^{'}=\{B(x_i,r_i)\}_{i\in\mathcal{I}^{'}}$ of pairwise disjoint
balls in $\mathcal{B}$ such that $$\bigcup_{B\in\mathcal{B}}B\subseteq\bigcup_{i\in\mathcal{I}^{'}}B(x_i,5r_i).$$
\end{lem}

\noindent{\it Proof of Proposition4.2.} We follow the argument in \cite{FH} for the condition of $L$-actions. Let $Z\subset X,s\geq
0,\varepsilon,\delta>0$, set $g=\chi_ Z,c_i\equiv 1$ in the definition of weighted topological entropy,
we have $W(Z,s,N,\epsilon)\leq M(Z,s,N,\epsilon)$ for $\forall N\in \mathbb{N}$.

Next, we prove $M(Z,s+\delta,N,\epsilon)\leq W(Z,s,N,\epsilon)$ for large enough $N$.

Let$\{(B_{n_i}(x_i,\epsilon),c_i)\}_{i\in \mathcal{I}}$ be a family so that $\mathcal{I}\subseteq \mathbb{N},
x_i\in X,0\leq
c_i<\infty,n_i\geq N,$ and
\begin{equation}\label{2.3}
    \sum_i c_i\chi_{B_i}\geq\chi_Z,
\end{equation}
here $B_i:=B_{n_i}(x_i,\epsilon)$. We claim that
\begin{equation}\label{2.4}
    M(Z,s+\delta,N,6\epsilon)\leq\sum_{i\in\mathcal{I}}c_i\bigg(\frac{1}{\lambda_{n_i}}\bigg)^s
\end{equation}
which implies $M(Z,s+\delta,N,6\epsilon)\leq W(Z,s,N,\epsilon)$.

We denote $$\mathcal{I}_n=\{i\in \mathcal{I}:n_i=n\},$$
and $$\mathcal{I}_{n,k}=\{i\in \mathcal{I}_n:i\leq k\}$$ for $n\geq N,k\in\mathbb{N}.$
We write $B_i:=B_{n_i}(x_i,\epsilon),5B_i:=B_{n_i}(x_i,5\epsilon)$ for $i\in\mathcal{I}$.
Obviously we may assume$B_i\neq B_j $ for $i\neq j$. For $t>0$, set
\begin{equation*}
    Z_{n,t}=\{x\in Z:\sum _{i\in \mathcal{I}_n}c_i\chi_{B_i}(x)>t\}\ \
\end{equation*}
and
\begin{equation*}
    Z_{n,t,k}=\{x\in Z:\sum _{i\in \mathcal{I}_{n,k}}c_i\chi_{B_i}(x)>t\}.
\end{equation*}
We divide the proof of (4.2) into the following three steps.

Step 1. For each $n\geq N, k\in \mathbb{N}$, and~$t>0$, there exists a finite set $\mathcal{J}_{n,k,t}\subseteq\mathcal{I}_{n,k} $
such that the ball $B_i,i\in\mathcal{J}_{n,k,t}$ are pairwise disjoint, $Z_{n,t,k}\subseteq\cup_{i\in\mathcal{J}_{n,k,t}}5B_i$,
and
\begin{equation*}
    \#(\mathcal{J}_{n,k,t})\bigg(\frac{1}{\lambda_n}\bigg)^s\leq\frac{1}{t}\sum_{i\in \mathcal{I}_{n,k}}c_i\bigg(\frac{1}{\lambda_n}\bigg)^s.
\end{equation*}
we will use the method of Federer \cite{FE}, also Mattila \cite{MA} for $L-$actions. Since $\mathcal{I}_{n,k}$ is
finite, by approximating the $c_i$'s from above, we may assume that each $c_i$ is a positive rational, and then by multiplying with
a common denominator we may assume that each $c_i$ is a positive integer. Let $m$ be the least integer with $m\geq t$.
Denote $\mathcal{B}=\{B_i,i\in\mathcal{I}_{n,k}\}$, and define ${u:\mathcal{B}\rightarrow\mathbb{Z}}$, by $u(B_i)=c_i$.
Since $B_i\neq B_j$, for $i\neq j$, so $u$ is well defined. We define by introduction integer-valued functions $v_0,v_1,\cdots,v_m$
on $\mathcal{B}$ and sub-families $\mathcal{B}_1,\mathcal{B}_{2},\ldots,\mathcal{B}_m$ of $\mathcal{B}$
starting with $v_0=u$. Using Lemma 4.1 repeatedly, we define inductively for $j=1,\cdots,m$, disjoint subfamilies $\mathcal{B}_{i}$
of $\mathcal{B}$ such that
\begin{equation*}
    \mathcal{B}_j\subset\{B\in\mathcal{B}:v_{j-1}(B)\geq1\},
\end{equation*}
\begin{equation*}
    Z_{n,k,t}\subseteq\cup_{B\in\mathcal{B}_j}5B
\end{equation*}
and the functions $v_j$ such that
$$v_j(B)=\left\{\begin{array}{ll}
v_{j-1}(B)-1,& \text{for}\ \ B\in\mathcal{B}_j ; \\
v_{j-1}(B),&
\text{for}\ \ B\in{\mathcal{B}\backslash\mathcal{B}_j\text{.}}\end{array}\right .$$
This is possible for $j<m$, $$Z_{n,k,t}\subseteq\{x:\sum_{B\in\mathcal{B}:B\ni x}v_j(B)\geq m-j\}, $$
whence every $x\in Z_{n,k,t}$ belongs to some ball $B\in \mathcal{B}$ with $v_j(B)\geq1$. Thus
\begin{eqnarray*}
\sum_{j=1}^{m}{\#(\mathcal{B}_{j})\bigg(\frac{1}{\lambda_n}\bigg)^{s}}&=&\sum_{j=1}^{m}{\sum_{B\in{\mathcal{B}}_{j}}(v_{j-1}(B)-v_{j}(B))\bigg(\frac{1}{\lambda_n}\bigg)^{s}}\\
&\leq & \sum_{B\in \mathcal{B}}{\sum_{j=1}^{m}(v_{j-1}(B)-v_{j}(B))\bigg(\frac{1}{\lambda_n}\bigg)^{s}} \\
&\leq& \sum_{B\in \mathcal{B}}u(B)\bigg(\frac{1}{\lambda_n}\bigg)^{s}=\sum_{i\in \mathcal{I}_{n,k}}c_i\bigg(\frac{1}{\lambda_n}\bigg)^{s}.
\end{eqnarray*}
Choose $j_0\in\{1,\cdots,m\}$ so that $\#(\mathcal{B}_{j_0})$ is the smallest. Then
\begin{eqnarray*}
\#(\mathcal{B}_{j_0})\bigg(\frac{1}{\lambda_n}\bigg)^{s}&\leq&\frac{1}{m}\sum_{i\in \mathcal{I}_{n,k}}c_i\bigg(\frac{1}{\lambda_n}\bigg)^{s}\leq  \frac{1}{t}\sum_{i\in \mathcal{I}_{n,k}}c_i\bigg(\frac{1}{\lambda_n}\bigg)^{s}.
\end{eqnarray*}
So $\mathcal{J}_{n,k,t}=\{i\in \mathcal{I}:B_i\in \mathcal{B}_{j_0}\}$ is desired.

Step 2. For each $n\in \mathbb{N}$ and $t>0$, we have
\begin{equation}\label{2.5}
    m(Z_{n,t},s+\delta,N,6\epsilon)\leq\frac{1}{\lambda_n^{\delta}t}\sum_{i\in \mathcal{I}_n}c_i\bigg(\frac{1}{\lambda_n}\bigg)^{s}.
\end{equation}

Assume $Z_{n,t}\neq\emptyset$, otherwise (4.3) is obvious. Since $Z_{n,k,t}\uparrow Z_{n,t}$, $Z_{n,k,t}\neq\emptyset$ for large
enough $k$. Let $\mathcal{J}_{n,k,t}$ be the sets constructed in Step 1. Then $\mathcal{J}_{n,k,t}\neq\emptyset$ for large
enough $k$. Set $E_{n,k,t}=\{x_i:i\in\mathcal{J}_{n,k,t}\}$. Note that the family of all non-empty compact subsets of $X$ is
compact with respect to Hausdorff distance(Federer[5,2.10.21]). It follows that there is a subsequence $(k_j)$ of natural
numbers and a non-empty compact set $E_{n,t}\subset X$ such that $E_{n,k_j,t}$ converges to $E_{n,t}$ in the Hausdorff distance
as $j\rightarrow\infty$. Since any two points in $E_{n,k,t}$ have a distance (with respect to $d_n$) not less than $\epsilon$,
so do the points in $E_{n,t}$. Thus $E_{n,t}$ is a finite set, moreover, $\#(E_{n,k_j,t})=\#(E_{n,t})$ when $j$ is large enough.
Hence $$\bigcup_{x\in E_{n,t}}B_n(x,5.5\epsilon)\supseteq
\bigcup_{x\in
E_{n,k_j,t}}B_n(x,5\epsilon)=\bigcup_{i\in\mathcal{J}_{n,k_j,t}}5B_i\supseteq
Z_{n,k_j,t},$$
when $j$ is large enough, and thus $\bigcup_{x\in E_{n,t}}B_n(x,6\varepsilon)\supseteq Z_{n,t}$. By the way, since
$\#{(E_{n,k_j,t})}=\#(E_{n,t})$ when $j$ is large enough, we have
\begin{equation*}
    \#(E_{n,t})\bigg(\frac{1}{\lambda_n}\bigg)^{s}\leq \frac{1}{t}\sum_{i\in
\mathcal{I}_n}c_i\bigg(\frac{1}{\lambda_n}\bigg)^{s}.
\end{equation*}
Therefore
\begin{equation*}
 M(Z_{n,t},s+\delta,N,6\epsilon)\leq\#(E_{n,t})\bigg(\frac{1}{\lambda_n}\bigg)^{s+\delta}\leq\frac{1}{{\lambda_n}^{\delta}t}\sum_{i\in \mathcal{I}_n}c_i\bigg(\frac{1}{\lambda_n}\bigg)^{s}.
\end{equation*}

Step 3. For any $t\in(0,1)$, we have
\begin{equation*}
 M(Z,s+\delta,N,6\epsilon)\leq\frac{1}{t}\sum_{i\in \mathcal{I}}c_i\bigg(\frac{1}{\lambda_{n_i}}\bigg)^{s},
\end{equation*}
which implies (4.2). In fact, fix $t\in(0,1)$. Then
$Z\subset\bigcup_{n=N}^\infty Z_{n,{\lambda_n}^{-\delta}t}$. Thus
by Proposition 4.1(i) and (4.3), we get
\begin{equation*}\begin{split}
M(Z,s+\delta,N,6\epsilon)&\leq\sum_{n=N}^{\infty}M(Z_{n,{\lambda_n}^{-\delta}t},s+\delta,N,6\epsilon)\\&
\leq\sum_{n=N}^{\infty}\frac{1}{t}{\sum_{i\in\mathcal{I}_n}c_i\bigg(\frac{1}{\lambda_n}\bigg)^s}=\frac{1}{t}\sum_{i\in\mathcal{I}}c_i\bigg(\frac{1}{\lambda_{n_i}}\bigg)^s.
\end{split}\end{equation*}
This completes the proof of the Proposition.

We will give a Frostman's lemma in dynamical system, which is important to our proof.
\begin{lem}
Suppose $K$ be a non-empty compact subset of $X$. Let $s\geq0,N\in\mathbb{N},\epsilon>0$. Set
$c:=W(K,s,N,\epsilon)>0.$ Then there exist a Borel probability measure $\mu$ on $X$ such that
$\mu(K)=1$ and $$\mu(B_n(x,\epsilon))\leq \frac{1}{c}\bigg(\frac{1}{\lambda_n}\bigg)^s.$$
\end{lem}
\begin{proof} Clearly $c<\infty$. We define a function $p$ on the space $C(X)$
of continuous real-valued functions on $X$ by
\begin{equation*}
    p(f)=\frac{1}{c}\text{W}(\chi_K\cdot f,s,N,\epsilon).
\end{equation*}
Let $\mathbf{1}\in C(X)$ denote the constant function $\mathbf{1}(x)\equiv 1$. It is easy to verify that

(1)$p(tf)=tp(f)$ for $f\in C(X)$ and $t\geq0$,

(2)$p(f+g)\leq p(f)+p(g)$ for $f,g\in C(X)$,

(3)$p(\mathbf{1})=1,0\leq P(f)\leq \parallel f\parallel_{\infty}$ for $f\in C(X)$, and $p(g)=0$ for $g\in C(X),g\leq0.$

By the Haha-Banach Theorem, we can extend the linear functional $t\rightarrow tp(1),t\in \mathbb{R}$, from the subspace
of constant functions to a linear functional $L:C(X)\rightarrow \mathbb{R}$ satisfying
\begin{equation*}
    L(\mathbf{1})=p(\mathbf{1})=1 \ \ \text{and}\ \ -p(-f)\leq L(f)\leq p(f), \text{for} \forall f\in C(X).
\end{equation*}
If $f\in C(X)$ with $f\geq 0$, then $p(-f)=0$ and so $L(f)\geq0$. Hence we can use the Riesz representation Theorem
to find a Borel probability measure $\mu$ on $X$ such that $L(f)=\int f d\mu$ for $f\in C(X)$.

Next, we prove $\mu(K)=1$. For any compact set $E\subset X\backslash K$, by Urysohn Lemma there exists $f\in C(X)$
such that $0\leq f\leq1,f(x)=1$ for $x\in E$ and $f(x)=0$ for $x\in K$. Then $f\cdot\chi_K=0$ and thus $p(f)=0$.
Hence $\mu(E)\leq L(f)\leq p(f)=0$. This shows $\mu(X\backslash K)=0,$ that is $\mu(K)=1$.

In the end, we prove $\mu(B_n(x,\epsilon))\leq\frac{1}{c}(\frac{1}{\lambda_n})^s$ for $\forall x\in X,n\geq N$.
In fact, for any compact set $E\subset B_n(x,\epsilon)$, by Urysohn lemma again, there is $f\in C(X)$, such that
$0\leq f\leq1, f(y)=1$ for $y\in E$ and $f(y)=0$ for $y\in X\backslash B_n(x,\epsilon)$. Then $\mu(E)\leq L(f)\leq p(f)$.
Since $\chi_K\cdot f\leq \chi_{B_n(x,\epsilon)}$, and $n\geq N$, we get $W(\chi_K\cdot f,s,N,\epsilon)\leq (\frac{1}{\lambda_n})^s$
and hence $p(f)\leq \frac{1}{c}(\frac{1}{\lambda_n})^s$. Therefore, we have $\mu(E)\leq\frac{1}{c}(\frac{1}{n})^s$. It follows that
\begin{equation*}
    \mu(B_n(x,\epsilon))=\sup\{\mu(E): E\subset B_n(x,\epsilon), \ \ \text{is compact}\}\leq\frac{1}{c}\bigg(\frac{1}{\lambda_n}\bigg)^s.
\end{equation*}
This completes the proof of the Lemma.
\end{proof}

\noindent{\textbf{Remark}: } There is a related slow entropy distribution principle. Using techniques in\cite{MW}, we have:
For any Borel set $E\subset X$ and Borel probability measure $\mu$ on $E$, if $\underline{h}_{\mu}^S(\mathcal{T},x)\leq s$ for all $x\in E$,
then $h^S_{top}(\mathcal{T,}E)\leq s$; if $\underline{h}_{\mu}^S(\mathcal{T},x)\geq s$ for all $x\in E,\mu(E)>0$, then $h^S_{top}(\mathcal{T},E)\geq  s$

\begin{thm}
Suppose $(X,\mathcal{T})$ be a TDS, $K\subset X$ be any non-empty compact subset. Then
\begin{equation*}
h^S_{top}(\mathcal{T},K)=\sup\{\underline{h}_\mu^S(\mathcal{T}):\mu\in \mathcal{M}(X),\mu(K)=1\}.
\end{equation*}
\end{thm}
\begin{proof} Firstly, we prove $h^S_{top}(\mathcal{T},K)\geq\underline{h}_\mu^S(\mathcal{T})$, for any $\mu\in \mathcal{M}(X),\mu(K)=1$.
We set
\begin{equation*}
\underline{h}_\mu^S(\mathcal{T},x,\epsilon)=\liminf\limits_{n\rightarrow\infty}-\frac{1}{\log \lambda_n}\log \mu(B_n(x,\epsilon))
\end{equation*}
for $x\in X,n\in\mathbb{N},\epsilon>0$. It's easy to see that $\underline{h}_\mu^S(\mathcal{T},x,\epsilon)$ is nonnegative and
increases as $\epsilon$ decreases. By the monotone convergence theorem ,we get
\begin{equation*}
\lim\limits_{\epsilon\rightarrow0}\int\underline{h}_\mu^S(\mathcal{T},x,\epsilon)d\mu(x)=\int\underline{h}_\mu^S(\mathcal{T},x)d\mu(x)=\underline{h}_\mu^S(\mathcal{T}).
\end{equation*}
Thus to show $h^S_{top}(\mathcal{T},K)\geq\underline{h}_\mu^S(\mathcal{T})$, we only to show $h^S_{top}(\mathcal{T},K)\geq\int\underline{h}_\mu^S(\mathcal{T},x,\epsilon)d\mu(x)$
for any $\epsilon>0$.

Now we fix $\epsilon>0, l\in\mathbb{N}$, set
$u_l=\min\{l,\int\underline{h}_\mu^S(\mathcal{T},x,\epsilon)d\mu(x)-\frac{1}{l}\}$,
then exist a Borel set $A_l\subset X,\mu(A_l)>0,N\in \mathbb{N}$
such that
$$\mu(B_n(x,\epsilon))\leq \bigg(\frac{1}{\lambda_n}\bigg)^{u_l}, \forall x\in A_l, n\geq N.\eqno{(4.4)}$$
Let $\{B_{n_i}(x_i,\epsilon/2)\}$ ba a finite or countable family such that $x_i\in X,n_i\geq N$ , and $K\cap A_l\subset \bigcup_iB_{n_i}(x_i,\epsilon/2)$.
We may as well assume that for each $i$, $B_{n_i}(x_i,\epsilon/2)\bigcap(K\cap A_l)\neq\emptyset$, and select $y_i\in B_{n_i}(x_i,\epsilon/2)\bigcap(K\cap A_l)$.
Then by (4.4),we have
\begin{equation*}\begin{split}
\sum_i\bigg(\frac{1}{\lambda_{n_i}}\bigg)^{u_l}&\geq \sum_i\mu(B_{n_i}(y_i,\epsilon))\\&\geq \sum_i\mu(B_{n_i}(x_i,\epsilon/2))
\\&\geq\mu(K\cap A_l)=\mu(A_l)>0.
\end{split}\end{equation*}
So, we get
\begin{equation*}\begin{split}
m(K,u_l)&\geq m(K,u_l,N,\epsilon/2)\\&\geq m(K\cap A_l,u_l,N,\epsilon/2)
\\&\geq\mu(A_l)>0.
\end{split}\end{equation*}
Therefore, $h^S_{top}(\mathcal{T},K)\geq u_l$. Letting $l\rightarrow\infty$, we get $h^S_{top}(\mathcal{T},K)\geq\int\underline{h}_\mu^S(\mathcal{T},x,\epsilon)d\mu(x)$.
Thus $h^S_{top}(\mathcal{T},K)\geq\underline{h}_\mu^S(\mathcal{T})$.

We next prove $h^S_{top}(\mathcal{T},K)\leq\{\underline{h}_\mu^S(\mathcal{T}):\mu\in
\mathcal{M}(X),\mu(K)=1\}$. We may as well assume
$h^S_{top}(\mathcal{T},K)>0$, otherwise the conclusion is obvious. By
Proposition 4.2, $h^S_{top}(\mathcal{T},K)=h_{top}^W(\mathcal{T},K)$. Suppose
$0<s<h^W_{top}(\mathcal{T},K)$, then there exists$\epsilon,N\in\mathbb{N},$
such that $c=W(K,s,N,\epsilon)>0$. By Lemma 4.2, there exists
$\mu\in\mathcal{M}(X),\mu(K)=1$, such that
$$\mu(B_n(x,\epsilon))\leq \frac{1}{c}\bigg(\frac{1}{\lambda_n}\bigg)^s$$
for any $x\in X,n\geq N$. And then $\underline{h}_\mu^S(T,x)\geq
s$ for each $x\in X$. Therefore,
$\underline{h}_\mu^S(\mathcal{T})\geq\int\underline{h}_\mu^S(\mathcal{T},x)d\mu(x)\geq
s.$ The proof is completed.
\end{proof}
\begin{cro}
Suppose $(X,T)$ be a TDS. Then
\begin{equation*}
h^S_{top}(T,X)=\sup_{\mu\in \mathcal{M}(X)}\underline{h}_\mu^S(T).
\end{equation*}
\end{cro}


\end{document}